\documentclass[a4paper,12pt,reqno]{amsart}
\usepackage{a4wide}
\usepackage{amsmath}
\usepackage{amssymb}
\usepackage{amsthm}
\usepackage{latexsym}
\usepackage{graphicx}
\usepackage[english]{babel}
              
\newtheorem{prop}[subsection]{Proposition}
\newtheorem{conj}[subsection]{Conjecture}
\newtheorem{teor}[subsection]{Theorem}

\newtheorem{cor} [subsection]{Corollary}
\theoremstyle{definition}

\theoremstyle{remark}

\def\hdepth{\operatorname{hdepth}}

\numberwithin{equation}{section}

\begin{document}

\title[Hdepth of the quotient ring of the edge ideal of a complete bipartite graph]
      {On the Hilbert depth of the quotient ring of the edge ideal of a complete bipartite graph}
\author[Andreea I.\ Bordianu, Mircea Cimpoea\c s]{Andreea I.\ Bordianu$^1$, Mircea Cimpoea\c s$^2$}  
\date{}



\footnotetext[1]{ \emph{Andreea I.\ Bordianu}, University Politehnica of Bucharest, Faculty of
Applied Sciences, 
Bucharest, 060042, E-mail: andreea.bordianu@stud.fsa.upb.ro}
\footnotetext[2]{ \emph{Mircea Cimpoea\c s}, National University of Science and Technology Politehnica Bucharest, Faculty of
Applied Sciences, 
Bucharest, 060042, Romania and Simion Stoilow Institute of Mathematics, Research unit 5, P.O.Box 1-764,
Bucharest 014700, Romania, E-mail: mircea.cimpoeas@upb.ro,\;mircea.cimpoeas@imar.ro}

\begin{abstract}
Let $n\geq m$ be two positive integers, $S_{n,m}=K[x_1,\ldots,x_n,y_1,\ldots,y_m]$ and $I_{n,m}=(x_iy_j\;:\;1\leq i\leq n,1\leq j\leq m)\subset S_{n,m}$ the edge ideal of a complete bipartite graph. Denote $h(n,m)=\hdepth(S_{n,m}/I_{n,m})$.
We prove that $h(n,m)\geq \left\lceil \frac{n}{2} \right\rceil$ and the equality holds if $m$ belong to a certain interval centered in
$\left\lceil \frac{n}{2} \right\rceil$. Also, we find some tight bounds for $h(n,n)$ and we prove several inequalities between
$h(n,m)$ and $h(n,m')$.
\end{abstract}

\maketitle

\section{Introduction}

Let $K$ be a field and let $S$ be a polynomial ring over $K$.
Let $M$ be a finitely generated graded $S$-module. The Hilbert depth of $M$, denoted by $\hdepth(M)$, is the 
maximal depth of a finitely generated graded $S$-module $N$ with the same Hilbert series as $M$; see
\cite{bruns,uli} for further details.
One would expect that it is easy to compute the Hilbert depth of a module, once its Hilbert function is known. 
But it turns out that even for the powers of the maximal ideal, the computation of the Hilbert depth leads to difficult 
combinatorial computations; see \cite{maxim}. In our opinion, the Hilbert depth invariant is unjustly underrated
in the literature and deserves a closer attention and more research.

In \cite[Theorem 2.4]{lucrare2} it was proved that if $0\subset I\subsetneq J\subset S$ are two squarefree monomial ideals, then
$$\hdepth(J/I)=\max\{q\;:\;\beta_k^q(J/I)=\sum_{j=0}^k (-1)^{k-j}\binom{q-j}{k-j}\alpha_j(J/I)\geq 0\text{ for all }0\leq k\leq q\},$$
where $\alpha_j(J/I)$ is the number of squarefree monomials of degree $j$ in $J\setminus I$, for $0\leq j\leq n$.

Let $n,m\geq 1$ be two integers. 
Using the above characterization of the Hilbert depth, in \cite{lucrare3}, it was showed that if 
$I_{n,m}=(x_iy_j\;:\;1\leq i\leq n,1\leq j\leq m)\subset S_{n,m}=K[x_1,\ldots,x_n,y_1,\ldots,y_m]$ ,
then $\hdepth(I_{n,m})=\left\lfloor \frac{n+m+2}{2} \right\rfloor$. Note that, $I_{n,m}$ is the edge ideal
 of the complete bipartite graph $K_{n,m}$. 

Although, the formula for $\hdepth(I_{n,m})$ is relatively easy to obtain, the problem of computing 
$\hdepth(S_{n,m}/I_{n,m})$ seems very difficult. In \cite[Theorem 2.6]{cipu}, it was proved that
$$\hdepth(S_{n,1}/I_{n,1}) \geq \left\lceil \frac{n}{2} \right\rceil + \left\lfloor \sqrt{n} \right\rfloor - 2,\text{ for all }n\geq 1.$$
Also, in \cite[Theorem ]{cipu}, it was proved that for any $\varepsilon>0$, there exists some integer $A=A(\varepsilon)\geq 0$ such that
$$\hdepth(S_{n,1}/I_{n,1})\leq \left\lceil \frac{n}{2} \right\rceil + \left\lfloor \varepsilon n \right\rfloor + A - 2.$$
As a direct consequence of these above results, we obtain the following asymptotic behavior:
$$\lim_{n\to\infty} \frac{1}{n}\hdepth(S_{n,1}/I_{n,1}) = \frac{1}{2}.$$
Our aim is to extend these results to the general case. We denote 
$$h(n,m):=\hdepth(S_{n,m}/I_{n,m}).$$
In Corollary \ref{minim} we note that 
$$h(n,m)\geq \left\lceil \frac{n}{2} \right\rceil,\text{ for all }1\leq m\leq n.$$
In Theorem \ref{t1}, we prove that $h(n,m) = \left\lceil \frac{n}{2} \right\rceil$, if
\begin{enumerate}
\item[(1)] $n=2s$ and $m=s+t$, such that $t\in \left( \frac{1-\sqrt{1+8s}}{2}, \frac{1+\sqrt{1+8s}}{2} \right)$.
\item[(2)] $n=2s+1$ and $m=s+1+t$, such that $t\in \left( \frac{1-\sqrt{9+16s}}{2}, \frac{1+\sqrt{9+16s}}{2} \right)$.
\end{enumerate} 
In Theorem \ref{t2}, we show that 
$$ \left\lfloor \frac{n}{2} + \sqrt{\left\lfloor \frac{n}{2}\right\rfloor \ln 2} \right\rfloor - 1 \leq h(n,n) \leq 
   \left\lfloor \frac{n+1}{2} + \sqrt{\left\lfloor \frac{n}{2}\right\rfloor \ln 2} \right\rfloor,\text{ for all }n\geq 2.$$
In Theorem \ref{t3}, we prove that
\begin{enumerate}
\item[(1)] $h(n,m)\geq h(n,m')$, for all $1\leq m\leq m'\leq \left\lfloor \frac{n}{2} \right\rfloor$.
\item[(2)] $h(n,m)\leq h(n,m')$, for all $\left\lceil \frac{n}{2} \right\rceil \leq m\leq m'\leq n$.
\item[(3)] $h(n,\left\lfloor \frac{n}{2} \right\rfloor)=h(n,\left\lceil \frac{n}{2} \right\rceil)=\left\lceil \frac{n}{2} \right\rceil$.
\item[(4)] $h(n,m)\geq h(n,n-m)$, for all $1\leq m\leq \left\lfloor \frac{n}{2} \right\rfloor$.
\item[(5)] $\left\lfloor \frac{n}{2} \right\rfloor\leq h(n,m)\leq h(n,1)$, for all $1\leq m\leq n$.
\end{enumerate}
As a consequence, in Corollary \ref{clim}, we show that
$$\lim_{n\to\infty} \frac{1}{n}h(n,m(n))=\frac{1}{2},$$
where $1\leq m(n)\leq n$, for all $n\geq 1$.

In Conjecture \ref{conj}, based on computer experiments, we proposed several inequalities. In Proposition \ref{pop}, we show
that if Conjecture \ref{conj} holds, then 
$$h(n,n)\geq \left\lfloor \frac{n-1}{2} + \sqrt{\left\lfloor \frac{n}{2}\right\rfloor \ln 2} \right\rfloor,$$
which improves the bound given in Theorem \ref{t2}.

\newpage
\section{Preliminaries}

Let $n\geq m\geq 1$ be two integers. 
The complete bipartite graph $K_{n,m}$ is the graph on the vertex set $$V(\mathcal K_{n,m})=\{x_1,\ldots,x_n,y_1,\ldots,y_m\},$$
with the edge set $$E(K_{n,m})=\{\{x_i,y_j\}\;:\;1\leq i\leq n,\;1\leq j\leq m\}.$$ 
Let $S_{n,m}=K[x_1,\ldots,x_n,y_1,\ldots,y_m]$. The edge ideal of $K_{n,m}$ is the ideal
$$I_{n,m}=(x_1,x_2,\ldots,x_n)\cap(y_1,\ldots,y_m)=(x_iy_j\;:\;1\leq i\leq n,\;1\leq j\leq m) \subset S_{n,m}.$$
First, we recall the following result.

\begin{teor}(See \cite[Theorem 2.9]{lucrare3})
We have that $$\hdepth(I_{n,m})=\left\lfloor \frac{n+m+2}{2} \right\rfloor.$$
\end{teor}

However, $\hdepth(S_{n,m}/I_{n,m})$ seems very difficult to compute, in general.
An upper bound is given by the following theorem:

\begin{teor}(\cite[Theorem 2.6]{lucrare3})
We have that $$\hdepth(S_{n,m}/I_{n,m})\leq \left\lceil n+m+\frac{1}{2}- \sqrt{2nm+\frac{1}{4}} \right\rceil.$$
\end{teor}

Also, we have the following result:

\begin{teor}\label{suc}(\cite[Theorem 2.7]{lucrare3})
We have that 
$$\hdepth(S_{n,m}/I_{n,m})=\max\{q\leq n+m\;:\; \binom{q-n}{2\ell}+\binom{q-m}{2\ell} \geq \binom{q}{2\ell},
\text{ for all } 1\leq \ell\leq \left\lfloor \frac{q}{2} \right\rfloor \}.$$
Moreover, $\hdepth(S/I) < m$, if $n \leq 2m-2$. Also, $m \leq \hdepth(S/I) \leq n-m+1$, if $n \geq 2m-1$.
\end{teor}

From the above theorem, we can easily find the following lower bound for $\hdepth(S_{n,m}/I_{n,m})$:

\begin{cor}\label{minim}
We have that $\hdepth(S_{n,m}/I_{n,m})\geq \left\lceil \frac{n}{2} \right\rceil$.
\end{cor} 

\begin{proof}
Let $q:=\left\lceil \frac{n}{2} \right\rceil$. For any $1\leq \ell\leq \left\lfloor \frac{q}{2} \right\rfloor$, we have that
\begin{align*}
\binom{q-n}{2\ell}+\binom{q-m}{2\ell} &\geq \binom{q-n}{2\ell} = \binom{n-q+2\ell-1}{2\ell} 
 = \binom{\left\lceil \frac{n}{2} \right\rceil+2\ell-1}{2\ell} \\  
& \geq \binom{\left\lceil \frac{n}{2} \right\rceil+1}{2\ell} \geq \binom{\left\lceil \frac{n}{2} \right\rceil}{2\ell}=\binom{q}{2\ell}.
\end{align*}
Hence, the conclusion follows from Theorem \ref{suc}.
\end{proof}

\section{Main results}

Let $S_{n,m}=K[x_1,\ldots,x_n,y_1,\ldots,y_m]$ and $I_{n,m}=(x_1,\ldots,x_n)\cap(y_1,\ldots,y_m)\subset S_{n,m}$. For convenience,
we denote $$h(n,m):=\hdepth(S_{n,m}/I_{n,m}).$$
Without further ado, we prove the following result:

\begin{teor}\label{t1}
We have that:
\begin{enumerate}
\item[(1)] If $n=2s$ and $m=s+t$, such that $t\in \left( \frac{1-\sqrt{1+8s}}{2}, \frac{1+\sqrt{1+8s}}{2} \right)$, then $h(n,m)=s$.
\item[(2)] If $n=2s+1$ and $m=s+1+t$, such that $t\in \left( \frac{1-\sqrt{9+16s}}{2}, \frac{1+\sqrt{9+16s}}{2} \right)$, then $h(n,m)=s+1$.
\end{enumerate}
\end{teor}

\begin{proof}
(1) According to Corollary \ref{minim}, we have that $h(n,m)\geq s$.
    Let $q=s+1$ and $\ell=1$. Then:
		$$ \binom{q-n}{2\ell}+\binom{q-m}{2\ell} = \binom{1-s}{2}+\binom{1-t}{2} < \binom{s+1}{2}, $$
		if and only if $$(1-s)(-s)+(1-t)(-t) < s(s+1),$$ which is equivalent to $$t^2-t-2s<0.$$ 
		Now, the conclusion follows from Theorem \ref{suc}.
		
(2) According to Corollary \ref{minim}, we have that $h(n,m)\geq s+1$.
    Let $q=s+2$ and $\ell=1$. Then:
		$$ \binom{q-n}{2\ell}+\binom{q-m}{2\ell} = \binom{1-s}{2}+\binom{1-t}{2} < \binom{s+2}{2}, $$
		if and only if $$(1-s)(-s)+(1-t)(-t) < (s+1)(s+2),$$ which is equivalent to $$t^2-t-4s-2<0.$$ Now, the conclusion
		follows from Theorem \ref{suc}.
\end{proof}

\begin{teor}\label{t2}
With the above notations, we have that:
$$ \left\lfloor \frac{n}{2} + \sqrt{\left\lfloor \frac{n}{2}\right\rfloor \ln 2} \right\rfloor - 1 \leq h(n,n) \leq 
   \left\lfloor \frac{n+1}{2} + \sqrt{\left\lfloor \frac{n}{2}\right\rfloor \ln 2} \right\rfloor,\text{ for all }n\geq 2.$$
\end{teor}

\begin{proof}
Let $q:=\left\lfloor \frac{n}{2} + \sqrt{\frac{n\ln 2}{2}} \right\rfloor-1$.
We prove the first inequality.
\begin{enumerate}
\item[(a)] Case $n=2k$ and $q-k=2s-1$. We have that $q=k+\left\lfloor \sqrt{k\ln 2} \right\rfloor-1$. 
           Since $\left\lfloor \sqrt{k\ln 2} \right\rfloor=2s$, it follows that
					 \begin{equation}\label{uc1}
					 k \geq \frac{4s^2}{\ln 2}.
					 \end{equation}            
           According to Theorem \ref{suc},
           in order to show that $h(n,n)\geq q$, it is enough to show that
					 \begin{equation}\label{qq1}
					 2\binom{n-q+2\ell-1}{2\ell} = 2\binom{k-2s+2\ell}{2\ell} \geq \binom{k+2s-1}{2\ell},
					 \text{ for all }1\leq \ell \leq \left\lfloor \frac{q}{2} \right\rfloor.
					 \end{equation}
					 In fact, it is enough to show \eqref{qq1} for $\ell=s$, that is 
					 \begin{equation}\label{qq2}
					 2\binom{k}{2s}\geq \binom{k+2s-1}{2s},
					 \end{equation}
					 as it follows easily from that all the other cases.
					
           Note that \eqref{qq2} is equivalent to
					 \begin{equation}\label{qq3}
					 \prod_{j=1}^{2s-1} \left( 1+ \frac{2s}{k-j} \right) \leq 2.
					 \end{equation}
					 Using the fact that $\left(1+\frac{1}{x}\right)^{x}<e$, for $x>0$, from \eqref{uc1} it follows that
					 \begin{align*}
					 \prod_{j=1}^{2s-1} \left( 1+ \frac{2s}{k-j} \right) & \leq \left( 1+ \frac{2s}{k-2s+1} \right)^{2s-1} 
					 \leq \left( 1+ \frac{2s\ln 2}{4s^2-(2s-1)\ln 2} \right)^{2s-1} \\
					  & = \left( \left( 1+ \frac{2s\ln 2}{4s^2-(2s-1)\ln 2} \right)^{\frac{4s^2-(2s-1)\ln 2}{2s\ln 2}} \right)^
						{\ln 2\cdot \frac{4s^2-2s}{4s^2-(2s-1)\ln 2}} \\
						& < \left(e^{\ln 2} \right)^{\frac{4s^2-2s}{4s^2-(2s-1)\ln 2}} = 2^{\frac{4s^2-2s}{4s^2-2s+1}}<2.				
					 \end{align*}
					 Hence \eqref{qq3} holds, as required.
					
\item[(b)] Case $n=2k$ and $q-k=2s$. We have that $q=k+\left\lfloor \sqrt{k\ln 2} \right\rfloor-1$. 
           Since $\left\lfloor \sqrt{k\ln 2} \right\rfloor=2s+1$, it follows that
					 \begin{equation}\label{uc11}
					 k \geq \frac{(2s+1)^2}{\ln 2}.
					 \end{equation}            
           According to Theorem \ref{suc},
           in order to show that $h(n,n)\geq q$, it is enough to show that
					 \begin{equation}\label{qq11}
					 2\binom{n-q+2\ell-1}{2\ell} = 2\binom{k-2s+2\ell-1}{2\ell} \geq \binom{k+2s}{2\ell},
					 \text{ for all }1\leq \ell \leq \left\lfloor \frac{q}{2} \right\rfloor.
					 \end{equation}
					 In fact, it is enough to show \eqref{qq1} for $\ell=s$, that is 
					 \begin{equation}\label{qq22}
					 2\binom{k-1}{2s}\geq \binom{k+2s}{2s},
					 \end{equation}
					 as it follows easily from that all the other cases.
					
           Note that \eqref{qq22} is equivalent to
					 \begin{equation}\label{qq33}
					 \prod_{j=1}^{2s} \left( 1+ \frac{2s+1}{k-j} \right) \leq 2.
					 \end{equation}
					 Using the fact that $\left(1+\frac{1}{x}\right)^{x}<e$, for $x>0$, from \eqref{uc11} it follows that
					 \begin{align*}
					 \prod_{j=1}^{2s} \left( 1+ \frac{2s+1}{k-j} \right) & \leq \left( 1+ \frac{2s+1}{k-2s} \right)^{2s} 
					 \leq \left( 1+ \frac{(2s+1)\ln 2}{(2s+1)^2- 2s \ln 2} \right)^{2s} \\
					  & = \left( \left( 1+ \frac{(2s+1)\ln 2}{(2s+1)^2- 2s \ln 2} \right)^{\frac{(2s+1)^2- 2s \ln 2}{(2s+1)\ln 2}} \right)^
						{\ln 2\cdot \frac{4s^2+2s}{4s^2+4s+1-2s\ln 2}} \\
						& < \left(e^{\ln 2} \right)^{\frac{4s^2+2s}{4s^2+2s+1}} = 2^{\frac{4s^2+2s}{4s^2+2s+1}}<2.				
					 \end{align*}
					 Hence \eqref{qq33} holds, as required.
	
\item[(c)] Case $n=2k+1$ and $q-k=2s$. We have that $q=k+\left\lfloor \sqrt{k\ln 2} - 0.5 \right\rfloor$. 
           Since $\left\lfloor \sqrt{k\ln 2} -0.5 \right\rfloor=2s$, it follows that
					 \begin{equation}\label{uc111}
					 k \geq \frac{(2s+0.5)^2}{\ln 2}.
					 \end{equation}            
           According to Theorem \ref{suc},
           in order to show that $h(n,n)\geq q$, it is enough to show that
					 \begin{equation}\label{qq111}
					 2\binom{n-q+2\ell-1}{2\ell} = 2\binom{k-2s+2\ell}{2\ell} \geq \binom{k+2s}{2\ell},
					 \text{ for all }1\leq \ell \leq \left\lfloor \frac{q}{2} \right\rfloor.
					 \end{equation}
					 In fact, it is enough to show \eqref{qq111} for $\ell=s$, that is 
					 \begin{equation}\label{qq222}
					 2\binom{k}{2s}\geq \binom{k+2s}{2s},
					 \end{equation}
					 as it follows easily from that all the other cases.
					
           Note that \eqref{qq222} is equivalent to
					 \begin{equation}\label{qq333}
					 \prod_{j=0}^{2s-1} \left( 1+ \frac{2s}{k-j} \right) \leq 2.
					 \end{equation}
					 Using the fact that $\left(1+\frac{1}{x}\right)^{x}<e$, for $x>0$, from \eqref{uc111} it follows that
					 \begin{align*}
					 \prod_{j=0}^{2s-1} \left( 1+ \frac{2s}{k-j} \right) & \leq \left( 1+ \frac{2s}{k-2s+1} \right)^{2s} 
					 \leq \left( 1+ \frac{2s\ln 2}{(2s+0.5)^2-(2s-1)\ln 2} \right)^{2s} \\
					  & =  \left( 1+ \frac{2s\ln 2}{4s^2+2s+0.25-2s\ln 2 + \ln 2} \right)^{2s} < 
						     \left( 1+ \frac{\ln 2}{2s} \right)^{2s}\\
						& = \left( \left( 1 + \frac{\ln 2}{2s} \right)^{\frac{2s}{\ln 2}} \right)^{\ln 2} < e^{\ln 2} = 2.				
					 \end{align*}
					 Hence \eqref{qq333} holds, as required.

\item[(d)] Case $n=2k+1$ and $q-k=2s+1$. We have that $q=k+\left\lfloor \sqrt{k\ln 2} - 0.5 \right\rfloor$. 
           Since $\left\lfloor \sqrt{k\ln 2} -0.5 \right\rfloor=2s+1$, it follows that
					 \begin{equation}\label{uc1111}
					 k \geq \frac{(2s+1.5)^2}{\ln 2}.
					 \end{equation}            
           According to Theorem \ref{suc},
           in order to show that $h(n,n)\geq q$, it is enough to show that
					 \begin{equation}\label{qq1111}
					 2\binom{n-q+2\ell-1}{2\ell} = 2\binom{k-2s+2\ell-1}{2\ell} \geq \binom{k+2s+1}{2\ell},
					 \text{ for all }1\leq \ell \leq \left\lfloor \frac{q}{2} \right\rfloor.
					 \end{equation}
					 In fact, it is enough to show \eqref{qq1111} for $\ell=s$, that is 
					 \begin{equation}\label{qq2222}
					 2\binom{k-1}{2s}\geq \binom{k+2s+1}{2s},
					 \end{equation}
					 as it follows easily from that all the other cases.
					
           Note that \eqref{qq2222} is equivalent to
					 \begin{equation}\label{qq2333}
					 \prod_{j=1}^{2s} \left( 1+ \frac{2s+2}{k-j} \right) \leq 2.
					 \end{equation}
					 Using the fact that $\left(1+\frac{1}{x}\right)^{x}<e$, for $x>0$, from \eqref{uc1111} it follows that
					 \begin{align*}
					 \prod_{j=1}^{2s} \left( 1+ \frac{2s+2}{k-j} \right) & \leq \left( 1+ \frac{2s+2}{k-2s} \right)^{2s} 
					 \leq \left( 1+ \frac{(2s+2)\ln 2}{(2s+1.5)^2- 2s \ln 2} \right)^{2s} \\
					  & =  \left( 1+ \frac{(2s+2)\ln 2}{4s^2+6s+2.25-2s\ln 2 } \right)^{2s} < 
						     \left( 1+ \frac{\ln 2}{2s} \right)^{2s}\\
						& = \left( \left( 1 + \frac{\ln 2}{2s} \right)^{\frac{2s}{\ln 2}} \right)^{\ln 2} < e^{\ln 2} = 2.				
					 \end{align*}
					 Hence \eqref{qq333} holds, as required.				
\end{enumerate}

Let $q:=\left\lfloor \frac{n+1}{2} + \sqrt{\left\lfloor \frac{n}{2}\right\rfloor \ln 2} \right\rfloor+1$.
We prove the second inequality.
\begin{enumerate}
\item[(a)] Case $n=2k$ and $q-k=2s+1$. We have that $q=k+ \left\lfloor \frac{1}{2} + \sqrt{k\ln 2} \right\rfloor+1$.
            Since $\left\lfloor \frac{1}{2} + \sqrt{k\ln 2} \right\rfloor=2s$, it follows that
					 $$2s \leq  \frac{1}{2} + \sqrt{k\ln 2}  < 2s+1,$$
					 from which we deduce that 
					 \begin{equation}\label{suc1}
					 \frac{\left( 2s - 0.5 \right)^2}{\ln 2} \leq k<\frac{\left( 2s + 0.5 \right)^2}{\ln 2}.
					 \end{equation}
					 According to Theorem \ref{suc},
           in order to show that $h(n,n) < q$, it is enough to show that there exists $1\leq \ell \leq \left\lfloor \frac{q}{2} \right\rfloor$
					 such that
					 \begin{equation}\label{xq1}
					 2\binom{n-q+2\ell-1}{2\ell} = 2\binom{k-2s+2\ell-2}{2\ell} < \binom{k+2s+1}{2\ell}.
					 \end{equation}
					 We claim that \eqref{xq1} holds for $\ell=s+1$, that is
					 \begin{equation}\label{xq2}
					 2\binom{k}{2s+2} < \binom{k+2s+1}{2s+2}.
					 \end{equation}
					 Note that \eqref{xq2} is equivalent to 
					 \begin{equation}\label{xq3}
					 \prod_{j=1}^{2s+1} \left(1 + \frac{2s+2}{k-j} \right) > 2.
					 \end{equation}
					 Using the fact that $\left(1+\frac{1}{x}\right)^{x+1}>e$, for $x>0$, from \eqref{suc1} it follows that 
					 \begin{align*}
					 \prod_{j=1}^{2s+1} \left(1 + \frac{2s+2}{k-j} \right) &> \left(1 + \frac{2s+2}{k-1} \right)^{2s+1}
					   \geq \left(1 + \frac{(2s+2)\ln 2}{(2s+0.5)^2 -\ln 2} \right)^{2s+1} \\
						& > \left(1 + \frac{(s+1)\ln 2}{2s^2+s} \right)^{2s+1}\\
						&  = 
						    \left( \left(1 + \frac{(s+1)\ln 2}{2s^2+s} \right)^{\frac{2s^2+s+(s+1)\ln 2}{(s+1)\ln 2}} \right)^
								{\ln 2\cdot \frac{(2s+1)(s+1)}{2s^2+s+(s+1)\ln 2}} \\
								& > \left(e^{\ln 2}\right)^{\frac{2s^2+3s+1}{2s^2+s+(s+1)\ln 2}}
								> 2.
						\end{align*}
						Hence, \eqref{xq3} holds, as required.
					 
\item[(b)] Case $n=2k$ and $q-k=2s$. We have that $q=k+ \left\lfloor \frac{1}{2} + \sqrt{k\ln 2} \right\rfloor+1$.
            Since $\left\lfloor \frac{1}{2} + \sqrt{k\ln 2} \right\rfloor=2s-1$, it follows that
					 $$2s-1 \leq  \frac{1}{2} + \sqrt{k\ln 2}  < 2s,$$
					 from which we deduce that 
					 \begin{equation}\label{suc11}
					 \frac{\left( 2s - 1.5 \right)^2}{\ln 2} \leq k<\frac{\left( 2s - 0.5 \right)^2}{\ln 2}.
					 \end{equation}
					 According to Theorem \ref{suc},
           in order to show that $h(n,n) < q$, it is enough to show that there exists $1\leq \ell \leq \left\lfloor \frac{q}{2} \right\rfloor$
					 such that
					 \begin{equation}\label{xq11}
					 2\binom{n-q+2\ell-1}{2\ell} = 2\binom{k-2s+2\ell-1}{2\ell} < \binom{k+2s}{2\ell}.
					 \end{equation}
					 We claim that \eqref{xq11} holds for $\ell=s$, that is
					 \begin{equation}\label{xq22}
					 2\binom{k-1}{2s} < \binom{k+2s}{2s}.
					 \end{equation}
					 Note that \eqref{xq22} is equivalent to 
					 \begin{equation}\label{xq33}
					 \prod_{j=1}^{2s} \left(1 + \frac{2s+1}{k-j} \right) > 2.
					 \end{equation}
					 Using the fact that $\left(1+\frac{1}{x}\right)^{x+1}>e$, for $x>0$, from \eqref{suc11} it follows that 
					 \begin{align*}
					 \prod_{j=1}^{2s} \left(1 + \frac{2s+1}{k-j} \right) &> \left(1 + \frac{2s+1}{k-1} \right)^{2s}
					   \geq \left(1 + \frac{(2s+1)\ln 2}{(2s-0.5)^2 -\ln 2} \right)^{2s} \\
						& > \left(1 + \frac{(2s+1)\ln 2}{4s^2-2s} \right)^{2s}\\
						&  = \left( \left(1 + \frac{(2s+1)\ln 2}{4s^2-2s} \right)^{\frac{4s^2-2s+(2s+1)\ln 2}{(2s+1)\ln 2}} \right)^
								{\ln 2\cdot \frac{4s^2+2s}{4s^2-2s+(2s+1)\ln 2}} \\
								& > \left(e^{\ln 2} \right)^{\frac{4s^2+2s}{4s^2-2s+(2s+1)\ln 2}}
								> 2.
						\end{align*}
						\normalsize Hence, \eqref{xq33} holds, as required.
	
\item[(c)] Case $n=2k+1$ and $q-k=2s+2$. We have that $q=k+ \left\lfloor \sqrt{k\ln 2} \right\rfloor+2$.
            Since $\left\lfloor  \sqrt{k\ln 2} \right\rfloor=2s$, it follows that
					 $$2s \leq \sqrt{k\ln 2}  < 2s+1,$$
					 from which we deduce that 
					 \begin{equation}\label{suc111}
					 \frac{4s^2}{\ln 2} \leq k<\frac{\left( 2s+1 \right)^2}{\ln 2}.
					 \end{equation}
					 According to Theorem \ref{suc},
           in order to show that $h(n,n) < q$, it is enough to show that there exists $1\leq \ell \leq \left\lfloor \frac{q}{2} \right\rfloor$
					 such that
					 \begin{equation}\label{xq111}
					 2\binom{n-q+2\ell-1}{2\ell} = 2\binom{k-2s+2\ell-2}{2\ell} < \binom{k+2s+2}{2\ell}.
					 \end{equation}
					 We claim that \eqref{xq11} holds for $\ell=s+1$, that is
					 \begin{equation}\label{xq222}
					 2\binom{k}{2s+2} < \binom{k+2s+2}{2s+2}.
					 \end{equation}
					 Note that \eqref{xq222} is equivalent to 
					 \begin{equation}\label{xq333}
					 \prod_{j=0}^{2s+1} \left(1 + \frac{2s+2}{k-j} \right) > 2.
					 \end{equation}
					 Using the fact that $\left(1+\frac{1}{x}\right)^{x+1}>e$, for $x>0$, from \eqref{suc11} it follows that 
					 \begin{align*}
					 \prod_{j=0}^{2s+1} \left(1 + \frac{2s+2}{k-j} \right) &> \left(1 + \frac{2s+2}{k} \right)^{2s+2}
					   \geq \left(1 + \frac{(2s+2) \ln 2}{(2s+1)^2} \right)^{2s+2} \\
						&  = \left( \left(1 + \frac{(2s+2)\ln 2}{(2s+1)^2} \right)^{\frac{(2s+1)^2+(2s+2)\ln 2}{(2s+2) \ln 2}} \right)^
								{\ln 2\cdot \frac{(2s+2)^2}{4s^2+4s+1+(2s+2)\ln 2}} \\
								& > \left(e^{\ln 2} \right)^{\frac{4s^2+8s+4}{4s^2+6s+3}} = 2^{\frac{4s^2+8s+4}{4s^2+6s+3}} > 2.
						\end{align*}						
						Hence, \eqref{xq333} holds, as required.

\item[(d)] Case $n=2k+1$ and $q-k=2s+1$. We have that $q=k+ \left\lfloor \sqrt{k\ln 2} \right\rfloor+2$.
            Since $\left\lfloor  \sqrt{k\ln 2} \right\rfloor=2s-1$, it follows that
					 $$2s-1 \leq \sqrt{k\ln 2}  < 2s,$$
					 from which we deduce that 
					 \begin{equation}\label{suc1111}
					 \frac{(2s-1)^2}{\ln 2} \leq k<\frac{4s^2}{\ln 2}.
					 \end{equation}
					 According to Theorem \ref{suc},
           in order to show that $h(n,n) < q$, it is enough to show that there exists $1\leq \ell \leq \left\lfloor \frac{q}{2} \right\rfloor$
					 such that
					 \begin{equation}\label{xq1111}
					 2\binom{n-q+2\ell-1}{2\ell} = 2\binom{k-2s+2\ell-1}{2\ell} < \binom{k+2s+1}{2\ell}.
					 \end{equation}
					 We claim that \eqref{xq1111} holds for $\ell=s$, that is
					 \begin{equation}\label{xq2222}
					 2\binom{k-1}{2s} < \binom{k+2s+1}{2s}.
					 \end{equation}
					 Note that \eqref{xq2222} is equivalent to 
					 \begin{equation}\label{xq3333}
					 \prod_{j=1}^{2s} \left(1 + \frac{2s+2}{k-j} \right) > 2.
					 \end{equation}
					 Using the fact that $\left(1+\frac{1}{x}\right)^{x+1}>e$, for $x>0$, from \eqref{suc11} it follows that 
					 \begin{align*}
					 \prod_{j=1}^{2s} \left(1 + \frac{2s+2}{k-j} \right) &> \left(1 + \frac{2s+2}{k} \right)^{2s}
					   \geq \left(1 + \frac{(s+1) \ln 2}{2s^2} \right)^{2s} \\
						&  = \left( \left(1 + \frac{(s+1)\ln 2}{2s^2} \right)^{\frac{2s^2+(s+1)\ln 2}{(s+1) \ln 2}} \right)^
								{\ln 2\cdot \frac{2s^2+2s}{2s^2+(s+1)\ln 2}} \\
								& > \left(e^{\ln 2} \right)^{\frac{2s^2+2s}{2s^2+(s+1)\ln 2}} = 2^{\frac{2s^2+2s}{2s^2+(s+1)\ln 2}} > 2.
						\end{align*}						
						Hence, \eqref{xq3333} holds, as required.
					
\end{enumerate}

\end{proof}

\begin{teor}\label{t3}
With the above notations, we have that:
\begin{enumerate}
\item[(1)] $h(n,m)\geq h(n,m')$, for all $1\leq m\leq m'\leq \left\lfloor \frac{n}{2} \right\rfloor$.
\item[(2)] $h(n,m)\leq h(n,m')$, for all $\left\lceil \frac{n}{2} \right\rceil \leq m\leq m'\leq n$.
\item[(3)] $h(n,\left\lfloor \frac{n}{2} \right\rfloor)=h(n,\left\lceil \frac{n}{2} \right\rceil)=\left\lceil \frac{n}{2} \right\rceil$.
\item[(4)] $h(n,m)\geq h(n,n-m)$, for all $1\leq m\leq \left\lfloor \frac{n}{2} \right\rfloor$.
\item[(5)] $\left\lfloor \frac{n}{2} \right\rfloor\leq h(n,m)\leq h(n,1)$, for all $1\leq m\leq n$.
\end{enumerate}
\end{teor}

\begin{proof}
(1) Let $q=h(n,m')$. Note that, according to Corollary \ref{minim}, $q\geq \left\lfloor \frac{n}{2} \right\rfloor$.
    From Theorem \ref{suc}, it follows that
    $$\binom{q-n}{2\ell}+\binom{q-m'}{2\ell}\geq \binom{q}{2\ell},\text{ for all }1\leq \ell\leq \left\lfloor \frac{q}{2} \right\rfloor.$$
		Since $m\leq m'$, it follows that $\binom{q-m}{2\ell}\geq \binom{q-m'}{2\ell}$, for all $1\leq \ell\leq \left\lfloor \frac{q}{2} \right\rfloor$.
		Hence $$\binom{q-n}{2\ell}+\binom{q-m}{2\ell}\geq \binom{q}{2\ell},\text{ for all }1\leq \ell\leq \left\lfloor \frac{q}{2} \right\rfloor.$$
		Thus, from Theorem \ref{suc}, it follows that $h(n,m)\geq q$, as required.
		
(2) Let $q=h(n,m')$. Note that, according to Corollary \ref{minim}, $q\geq \left\lfloor \frac{n}{2} \right\rfloor$. On the other hand, according
    to Theorem \ref{suc}, we have that $q\leq m$ and
		$$\binom{q-n}{2\ell}+\binom{q-m}{2\ell}\geq \binom{q}{2\ell},\text{ for all }1\leq \ell\leq \left\lfloor \frac{q}{2} \right\rfloor.$$
		Since $m'\geq m\geq q$, it follows that 
		$$\binom{q-m}{2\ell} = \binom{m-q+2\ell-1}{2\ell} \leq \binom{m'-q+2\ell-1}{2\ell} = \binom{q-m'}{2\ell},\text{ for all }1\leq \ell\leq \left\lfloor \frac{q}{2} \right\rfloor.$$
		Therefore, from Theorem \ref{suc}, we get $h(n,m')\geq q$, as required.
		
(3) It follows immediately from Theorem \ref{t1}.

(4) Let $q=h(n,n-m)$. From Corollary \ref{minim} and Theorem \ref{suc} we have that 
    $\left\lfloor \frac{n}{2} \right\rfloor \leq q\leq n-m$ and
		\begin{equation}\label{cocos}
    \binom{q-n}{2\ell}+\binom{q-n+m}{2\ell}\geq \binom{q}{2\ell},\text{ for all }1\leq \ell\leq \left\lfloor \frac{q}{2} \right\rfloor.
	  \end{equation}
		In order to complete the proof, we have to show that
		\begin{equation}\label{wish1}
		\binom{q-n}{2\ell}+\binom{q-m}{2\ell}\geq \binom{q}{2\ell},\text{ for all }1\leq \ell\leq \left\lfloor \frac{q}{2} \right\rfloor.
		\end{equation}
		We choose an integer $\ell$ with $1\leq \ell\leq \left\lfloor \frac{q}{2} \right\rfloor$.
		
		If $2\ell\geq 2q-n+1$, then 
		$\binom{q-n}{2\ell} = \binom{n-q+2\ell-1}{2\ell}\geq \binom{q}{2\ell}$ and, therefore, \eqref{wish1} holds.
		Now, assume that $2\ell\leq 2q-n$. Then
		$$\binom{q-n+m}{2\ell} = \binom{n-m-q+2\ell-1}{2\ell} \leq \binom{q-m-1}{2\ell}\leq \binom{q-m}{2\ell}.$$
		From \eqref{cocos} it follows that \eqref{wish1} holds also.
		Hence, the conclusion follows from Theorem \ref{suc}.
    
(5) It follows from Corollary \ref{minim}, (1) and (4).
\end{proof}

We denote by $\mathbb N=\{1,2,3,\ldots\}$, the set of positive integers.

\begin{cor}\label{clim}
Let $m:\mathbb N \to \mathbb N$ be a function, such that $m(n)\leq n$, for all $n\geq 1$. Then
$$\lim_{n\to\infty} \frac{1}{n}h(n,m(n))=\frac{1}{2}.$$
\end{cor}

\begin{proof}
According to \cite[Corollary 2.7]{cipu}, we have that 
$$\lim_{n\to\infty} \frac{1}{n}h(n,1)=\frac{1}{2}.$$
On the other hand, from Theorem \ref{t3}(5), we have that 
$$\left\lfloor \frac{n}{2} \right\rfloor\leq h(n,m(n))\leq h(n,1),\text{ for all }1\leq m\leq n.$$
The required conclusion follows immediately.
\end{proof}

\section{A conjecture and a sharper bound for $h(n,n)$}

Based on computer experiments, we proposed the following:

\begin{conj}\label{conj}
Let $s$ be a positive integer. Then:
\begin{enumerate}

\item[(a)] $\prod\limits_{j=1}^{2s} \left( 1 + \frac{2s+1}{k-j} \right) \leq 2$, for $k\geq \frac{(2s+0.5)^2}{\ln(2)} $. 

\item[(b)] $\prod\limits_{j=1}^{2s+1} \left( 1 + \frac{2s+2}{k-j} \right) \leq 2$, for $k \geq \frac{(2s+1.5)^2}{\ln(2)}$. 

\item[(c)] $\prod\limits_{j=0}^{2s-1} \left( 1 + \frac{2s}{k-j} \right) \leq 2$, for $k \geq \frac{4s^2}{\ln(2)}$.

\item[(d)] $\prod\limits_{j=2}^{2s+1} \left( 1 + \frac{2s+2}{k-j} \right) \leq 2$, for $k \geq \frac{(2s+1)^2}{\ln(2)}$.

\end{enumerate}
\end{conj}

\begin{prop}\label{pop}
If Conjecture \ref{conj} holds, then 
$$h(n,n)\geq \left\lfloor \frac{n-1}{2} + \sqrt{\left\lfloor \frac{n}{2}\right\rfloor \ln 2} \right\rfloor.$$
\end{prop}

\begin{proof}
The idea of the proof is similar to the proof of Theorem 2. 
We let $q=\left\lfloor \frac{n-1}{2} + \sqrt{\left\lfloor \frac{n}{2}\right\rfloor \ln 2} \right\rfloor$ and we consider several cases:
\begin{enumerate}
\item[(a)] $n=2k$ and $\left\lfloor -0.5+ \sqrt{k \ln 2} \right\rfloor=2s$. Note that $k\geq \frac{(2s+0.5)^2}{\ln(2)}$.
           We have $q=k+2s$. According to Theorem \ref{suc},
           in order to show that $h(n,n)\geq q$, it is enough to show that
					 \begin{equation}\label{eq1}
					 2\binom{n-q+2\ell-1}{2\ell} = 2\binom{k-2s+2\ell-1}{2\ell} \geq \binom{k+2s}{2\ell},
					 \text{ for all }1\leq \ell \leq \left\lfloor \frac{q}{2} \right\rfloor.
					 \end{equation}
					 In fact, it is enough to show \eqref{eq1} for $\ell=s$, that is
					 $$
					 2\binom{k-1}{2s} \geq \binom{k+2s}{2s},
					 $$
					 which is equivalent to $\prod\limits_{j=1}^{2s} \left( 1 + \frac{2s+1}{k-j} \right) < 2$.
					 Since $k\geq \frac{(2s+0.5)^2}{\ln(2)}$, the conclusion follows from Conjecture \ref{conj}(a).
					
\item[(b)] $n=2k$ and $\left\lfloor -0.5+ \sqrt{k \ln 2} \right\rfloor=2s+1$. Note that $k \geq \frac{(2s+1.5)^2}{\ln(2)}$.
           We have $q=k+2s+1$. According to Theorem \ref{suc},
           in order to show that $h(n,n)\geq q$, it is enough to show that
					 \begin{equation}\label{eq2}
					 2\binom{n-q+2\ell-1}{2\ell} = 2\binom{k-2s+2\ell-2}{2\ell} \geq \binom{k+2s+1}{2\ell},
					 \text{ for all }1\leq \ell \leq \left\lfloor \frac{q}{2} \right\rfloor.
					 \end{equation}
					 In fact, it is enough to show \eqref{eq2} for $\ell=s+1$, that is
					 $$
					 2\binom{k}{2s+2} \geq \binom{k+2s+1}{2s+2},
					 $$
					 which is equivalent to $\prod\limits_{j=1}^{2s+1} \left( 1 + \frac{2s+2}{k-j} \right) \leq 2$.
					 Since $k \geq \frac{(2s+1.5)^2}{\ln(2)}$, the conclusion follows from Conjecture \ref{conj}(b).
					
\item[(c)] $n=2k+1$ and $\left\lfloor  \sqrt{k \ln 2} \right\rfloor=2s$. Note that $k \geq \frac{4s^2}{\ln(2)}$.
           We have $q=k+2s$. According to Theorem \ref{suc},
           in order to show that $h(n,n)\geq q$, it is enough to show that
					 \begin{equation}\label{eq3}
					 2\binom{n-q+2\ell-1}{2\ell} = 2\binom{k-2s+2\ell}{2\ell} \geq \binom{k+2s}{2\ell},
					 \text{ for all }1\leq \ell \leq \left\lfloor \frac{q}{2} \right\rfloor.
					 \end{equation}
					 In fact, it is enough to show \eqref{eq3} for $\ell=s$, that is
					 $
					 2\binom{k}{2s} \geq \binom{k+2s}{2s},
					 $
					 which is equivalent to $\prod\limits_{j=0}^{2s-1} \left( 1 + \frac{2s}{k-j} \right) \leq 2$.
					 Since $k \geq \frac{4s^2}{\ln(2)}$, the conclusion follows from Conjecture \ref{conj}(c).

\item[(d)] $n=2k+1$ and $\left\lfloor  \sqrt{k \ln 2} \right\rfloor=2s+1$. Note that $k \geq \frac{(2s+1)^2}{\ln(2)}$.
           We have $q=k+2s+1$. According to Theorem \ref{suc},
           in order to show that $h(n,n)\geq q$, it is enough to show that
					 \begin{equation}\label{eq4}
					 2\binom{n-q+2\ell-1}{2\ell} = 2\binom{k-2s+2\ell-1}{2\ell} \geq \binom{k+2s+1}{2\ell},
					 \text{ for all }1\leq \ell \leq \left\lfloor \frac{q}{2} \right\rfloor.
					 \end{equation}
					 In fact, it is enough to show \eqref{eq4} for $\ell=s+1$, that is
					 $$
					 2\binom{k+1}{2s+2} \geq \binom{k+2s+1}{2s+2},
					 $$
					 which is equivalent to $\prod\limits_{j=2}^{2s+1} \left( 1 + \frac{2s+2}{k-j} \right) \leq 22$.
					 Since $k \geq \frac{(2s+1)^2}{\ln(2)}$, the conclusion follows from Conjecture \ref{conj}(d).
\end{enumerate}
\end{proof}

Note that the bound given in Proposition \ref{pop} is sharper than the one from Theorem \ref{t2}.

\subsection*{Acknowledgment}

We would like to express our thanks to Professor Mihai Cipu for valuables discussions which helped us to improve this manuscript.






\begin{thebibliography}{9}

\bibitem{lucrare2} S.\ B\u al\u anescu, M.\ Cimpoea\c s, C.\ Krattenthaller, 
                   \emph{On the Hilbert depth of monomial ideals}, to appear in Comm. Algebra, arXiv:2306.09450v4 (2024).

\bibitem{lucrare3} S.\ B\u al\u anescu, M.\ Cimpoea\c s, 
                   \emph{On the Hilbert depth of certain monomial ideals and applications}, U.P.B. Sci. Bull., Series A
									 \textbf{87(4)} (2025), 53--66.

\bibitem{cipu} S.\ B\u al\u anescu, M.\ Cimpoea\c s, M.\ Cipu, 
\emph{On the Hilbert depth of the quotient ring of the edge ideal of a star graph}, arXiv:2501.16742 (2025).

\bibitem{bruns} W.\ Bruns, C.\ Krattenthaler, J.\ Uliczka, \emph{Stanley decompositions and Hilbert depth in the Koszul complex}, 
                J. Commut. Algebra \textbf{2(3)} (2010), 327--357.


\bibitem{maxim} W.\ Bruns, C.\ Krattenthaler and J.\ Uliczka, {\it Hilbert depth of powers of the maximal ideal}, 
                Commutative algebra and its connections to geometry, Contemp. Math. 555 (2011), 1--12.

\bibitem{uli} J.\ Uliczka, \emph{Remarks on Hilbert series of graded modules over polynomial rings}, 
              Manuscr. Math. \textbf{132} (2010), 159--168.
	
\end{thebibliography}
\end{document}